\documentclass[reqno]{amsart}

\usepackage{a4wide}
\usepackage[russian,english]{babel}
\usepackage[T2A]{fontenc}
\usepackage[utf8]{inputenc} 
\usepackage{amsfonts}
\usepackage{amssymb, amsthm, amscd}
\usepackage{amsmath}
\usepackage{mathtools}
\usepackage{needspace}
\usepackage{lipsum}
\usepackage{comment}
\usepackage{etoolbox}
\usepackage{cmap}
\usepackage[pdftex]{graphicx}
\usepackage[unicode]{hyperref}
\usepackage[matrix,arrow,curve]{xy}
\usepackage[usenames,dvipsnames]{xcolor}
\usepackage{colortbl}
\usepackage{textcomp}
\usepackage[hyper]{amsbib}

\newcommand{\CP}{\mathbb{C}\mathrm{P}}
\newcommand{\R}{\mathbb{R}}

\renewcommand{\C}{\mathcal{C}}
\renewcommand{\L}{\mathcal{L}}
\newcommand{\N}{\mathbf{N}}
\newcommand{\M}{\mathcal{M}}
\renewcommand{\H}{\mathcal{H}}
\newcommand{\A}{\mathcal{A}}
\newcommand{\eps}{\varepsilon}
\newcommand{\q}{\;\,}
\newcommand{\grad}{\mathop{\mathrm{grad}}\nolimits}
\newcommand{\E}{\mathcal{E}}
\newcommand{\ph}{\varphi}

\newcommand{\Poly}{\mathop{\mathrm{Poly}}\nolimits}
\newcommand{\Isom}{\mathop{\mathrm{Isom}}\nolimits}
\newcommand{\AffHull}{\mathop{\mathrm{AffineHull}}\nolimits}
\newcommand{\Hess}{\mathop{\mathrm{Hess}}\nolimits}
\newcommand{\Jac}{\mathop{\mathrm{Jac}}\nolimits}

\renewcommand{\tilde}{\widetilde}

\theoremstyle{plain}
\newtheorem{thm}{Theorem}[section]
\newtheorem{oldthm}[thm]{Theorem}
\newtheorem{lem}[thm]{Lemma}
\newtheorem{oldlem}[thm]{Lemma}
\newtheorem{prop}[thm]{Proposition}

\theoremstyle{definition}
\newtheorem{defn}[thm]{Definition}
\newtheorem{cor}[thm]{Corollary}

\theoremstyle{remark}
\newtheorem{rem}[thm]{Remark}

\AtBeginEnvironment{thm}{\begin{samepage}}
\AtEndEnvironment{thm}{\end{samepage}}
\AtBeginEnvironment{oldthm}{\begin{samepage}}
\AtEndEnvironment{oldthm}{\end{samepage}}
\AtBeginEnvironment{lem}{\begin{samepage}}
\AtEndEnvironment{lem}{\end{samepage}}
\AtBeginEnvironment{st}{\begin{samepage}}
\AtEndEnvironment{st}{\end{samepage}}
\AtBeginEnvironment{crit}{\begin{samepage}}
\AtEndEnvironment{crit}{\end{samepage}}
\AtBeginEnvironment{ax}{\begin{samepage}}
\AtEndEnvironment{ax}{\end{samepage}}
\AtBeginEnvironment{defn}{\begin{samepage}}
\AtEndEnvironment{defn}{\end{samepage}}
\AtBeginEnvironment{cor}{\begin{samepage}}
\AtEndEnvironment{cor}{\end{samepage}}
\AtBeginEnvironment{rem}{\begin{samepage}}
\AtEndEnvironment{rem}{\end{samepage}}
\AtBeginEnvironment{note}{\begin{samepage}}
\AtEndEnvironment{note}{\end{samepage}}
\AtBeginEnvironment{prop}{\begin{samepage}}
\AtEndEnvironment{prop}{\end{samepage}}

\title{Oriented Area as a Morse Function on Polygon Spaces}
\author{Daniil Mamaev}
\date{\today (Last Typeset)}
\address{Chebyshev Laboratory, St. Petersburg State University, 14th Line V.O., 29, Saint Petersburg 199178 Russia}
\email{dan.mamaev@gmail.com}
\thanks{The research is supported by «Native towns», a social investment program of PJSC «Gazprom Neft»}

\begin{document}
\selectlanguage{english}
\begin{abstract}
We study polygon spaces arising from planar configurations of necklaces with some of the beads fixed and some of the beads sliding freely. These spaces include configuration spaces of flexible polygons and some other natural polygon spaces. We characterise critical points of the oriented area function in geometric terms and give a formula for the Morse indices. Thus we obtain a generalisation of isoperimetric theorems for polygons in the plane.
\end{abstract}

\maketitle

\section{Preliminaries: necklaces, configuration spaces, and oriented area function}

Suppose one has a closed string with a number of labelled beads, a {\itshape necklace}. Some of the beads are fixed and some can slide freely (although the beads never pass through one another). Having the necklace in hand, one can try to put it on the plane (self-intersections are allowed) in such a way that the string is strained between every two consecutive beads. We will call this a {\itshape (strained planar) configuration of the necklace}. The space of all configurations (up to rotations and translations) of a given necklace, called {\itshape configuration space of the necklace}, together with the oriented area function on it is the main object of the present paper.

Let us now be precise. Given a tuple $(n_1, \ldots, n_k)$ of positive integers and a tuple $(L_1, \ldots, L_k)$ of positive reals, we define {\itshape a necklace $\N$} to be a tuple $\big((n_1, L_1), \ldots, (n_k, L_k)\big)$ interpreted as follows: 
\begin{itemize}
\item the necklace has the total of $n = n(\N) = n_1 + \ldots + n_k$ beads on it;
\item $k$ of the beads are fixed and numbered by the index $j = 1, \ldots, k$ in counter-clockwise order, the index $j$ is considered to be cyclic (that is, $j = 6k + 5$ is the same as $j = 5$);
\item there are $(n_j - 1)$ freely sliding beads between the $j$-th and the $(j + 1)$-th fixed bead;
\item the string has the total length of $L = L(\N) = L_1 + \ldots + L_k$;
\item the length of the string between the $j$-th and the $(j + 1)$-th fixed bead is equal to $L_j$.
\end{itemize}

We fix the language we will talk about polygons in the present paper.
\begin{itemize}
\item {\itshape A planar $n$-gon} is a collection of $n$ (labelled) points $(p_1, \ldots, p_n)$ in the Euclidian plane $\R^2$.
\item {\itshape The space of all planar $n$-gons $\Poly_n$} is thus just $\left(\R^2\right) ^ n$.
\item The {\itshape sides} of a polygon $P = (p_1, \ldots, p_n)$ are the segments $p_ip_{i + 1}$ for $i = 1, \ldots, n$, the {\itshape length of the $i$-th side} is $l_i = l_i(P) = |p_ip_{i + 1}|$. Note that the index $i = 1, \ldots, n$ is cyclic (that is, $i = 10n + 3$ is the same as $i = 3$).
\end{itemize} 

To avoid the messy indices, we introduce some additional notation associated with a necklace ${\N = \big((n_1, L_1), \ldots, (n_k, L_k)\big)}$. For index $j = 1, \ldots, k$
\begin{itemize}
\item we denote by $j^*$ the set of {\itshape indices, corresponding to the $j$-th piece~of~$\N$}:
\begin{equation} \label{definition of j^*}
j^* =  \{n_1 + \ldots + n_{j - 1} + 1, \ldots, n_1 + \ldots + n_j\};
\end{equation}
\item we introduce a function $\L_j\colon \Poly_n \to \R$, {\itshape the total length of the sides of a polygon, corresponding to the $j$-th piece of $\N$}, that is
\begin{equation} \label{definition of L_j}
\L_j(P) = \sum_{i \in j^*} l_i(P).
\end{equation}
\end{itemize}

And now we are ready to give the following
\begin{defn} \label{definition of a strained planar configuration of a necklace}
 {\itshape A (strained planar)  configuration of necklace ${\N = \big((n_1, L_1), \ldots, (n_k, L_k)\big)}$} is a polygon $P \in \Poly_n$ with $\L_j(P) = L_j$ for all $j = 1, \ldots, k$. 
\end{defn}

\begin{figure}
\label{example of a polygon}
\includegraphics[scale=0.3]{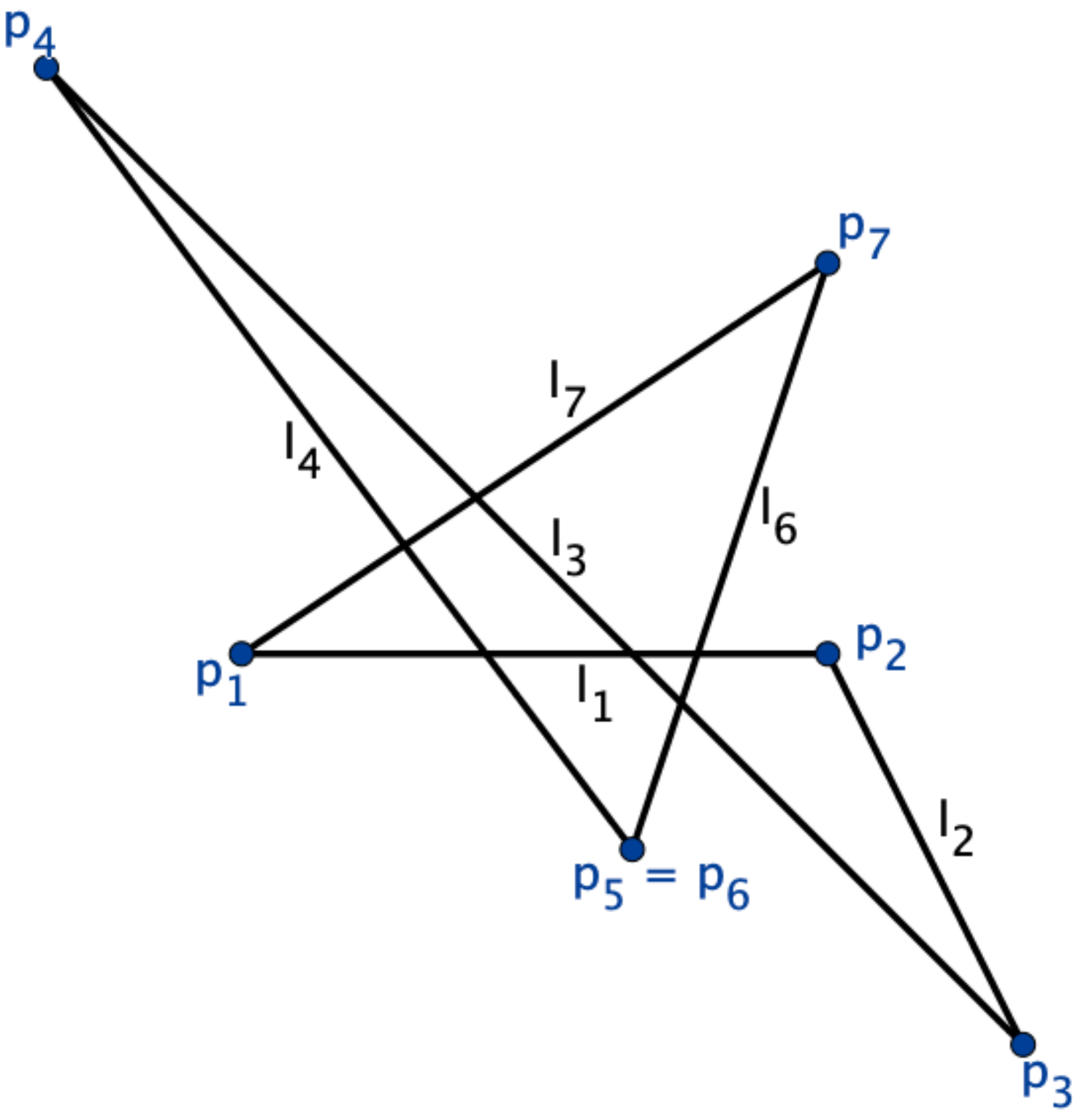}
\caption{A polygon in $\M\big((2, l_1 + l_2), (1, l_3), (4, l_4 + l_6 + l_7)\big)$}
\end{figure}

All configurations of necklace $\N = \big((n_1, L_1), \ldots, (n_k, L_k)\big)$ modulo translations and rotations form a space $\M(\N) = \M\big((n_1, L_1), \ldots, (n_k, L_k)\big)$ called {\itshape configuration space of necklace~$\N$}. More formally,
\begin{itemize}
\item Consider all the strained planar configurations of $\N$:
\begin{equation} \label{definition of the set of all configurations of a necklace}
\tilde{\M}(\N) = \tilde{\M}\big((n_1, L_1), \ldots, (n_k, L_k)\big) = \left\{P \in \Poly_n \left| \begin{aligned}
\L_j(P) = L_j \text{ for } j = 1, \ldots, k
\end{aligned} \right.\right\}. 
\end{equation}
\item The group $\Isom_+\left(\R^2\right)$ of orientation-preserving isometries of the Euclidean plane $\R^2$ acts diagonally on the space of all planar $n$-gons $\Poly_n = \left(\R ^ 2\right) ^ n$.
\item $\tilde \M(\N)$ is invariant under the action.
\item {\itshape The configuration space of the necklace $\N$} is the space of orbits:
\begin{equation} \label{definition of the configuration space of a necklace}
\M(\N) = \left.\tilde{\M}(\N)\right/\Isom_+(\R^2). 
\end{equation}
\end{itemize}

\begin{defn} \label{definition of oriented area}
{\itshape Oriented area} $\A$ of an $n$-gon ${P = ((x_1, y_1), \ldots, (x_n, y_n)) \in \left(\R ^ 2\right) ^ n}$ is defined by
\begin{equation}
\A(P) = 
\frac{1}{2} \begin{vmatrix}
x_1 & x_2\\ y_1 & y_2
\end{vmatrix} +
\frac{1}{2} \begin{vmatrix}
x_2 & x_3\\ y_2 & y_3
\end{vmatrix} + \ldots + 
\frac{1}{2} \begin{vmatrix}
x_n & x_1\\ y_n & y_1
\end{vmatrix}. 
\end{equation}
\end{defn}
Oriented area is preserved by the action of $\Isom_+(\R^2)$ and thus gives rise to a well defined continuous function on $\left(\R^2\right)^n/\Isom_+(\R^2)$ hence on all of $\M(\N)$. We will denote these functions by the same letter $\A$. The study of critical points (i. e. the solutions of $d\A (P) = 0$) of $\A \colon \M(\N) \to \R$ is the substance of the present paper. 

The paper is organised as follows. In Section~\ref{Special cases} we review previously studied extreme cases: polygonal linkages (the `all beads are fixed' case) and polygons with fixed perimeter (the `one bead is fixed' case, which is clearly the same as `none of the beads are fixed' case). In Section~\ref{Discussion on singular locus of the configuration space} we discuss the regularity properties of configuration spaces of necklaces. In the subsequent sections we study only non-singular part of configuration space. In Section~\ref{Main results} we give a geometric description of critical points of oriented area in the general case (Theorem~\ref{critical points}) and deduce a formula for their Morse indices (Theorem~\ref{Morse index}) from Lemmata~\ref{orthogonality of cyclic and linkages},~\ref{orthogonality of cyclic pieces},~\ref{Morse index for cyclic}. In Section~\ref{Orthogonality with respect to the Hessian of oriented area} the auxiliary Lemmata~\ref{orthogonality of cyclic and linkages}~and~\ref{orthogonality of cyclic pieces} concerning orthogonality of certain spaces with respect to the Hessian form of the oriented area function are proven. In Section~\ref{Addenda} we discuss the `two consecutive beads are fixed' case and give a proof of Lemma~\ref{Morse index for cyclic}. 

\subsection*{Acknowledgement}
This research is a continuation of my bachelor thesis in which the case of configuration spaces of necklaces with two consecutive fixed beads was studied. I am deeply indebted to Gaiane Panina for posing the problem and supervising my research. I am also thankful to Joseph Gordon and Alena Zhukova for fruitful discussions and to Nathan Blacher for his valuable comments on the linguistic quality of the paper.

\section{An overview of existing results}
\label{Special cases}
\subsection{Configuration spaces of polygonal linkages} In the notation of the present paper these are the spaces ${\M((1, l_1), \ldots, (1, l_n))}$, i. e. the spaces $\M(\N)$ for necklaces $\N$ with all beads being fixed. These spaces are studied in many aspects (see e. g. \cite{CD17} or \cite{F08} for a thorough survey). On the side of studying oriented area on these spaces, the first general fact about its critical points was noticed by Thomas Banchoff (unpublished), reproved by Khimshiashvili and Panina \cite{KP08} (their technique required some non-degeneracy assumptions) and then reproved again by Leger \cite{L18} in full generality.

\begin{samepage}
\begin{oldthm}[{\bfseries Critical configurations in the `all beads are fixed' case})\\(Bunchoff, Khimshiashvili, Leger, and Panina]\label{critical points for linkages}~\\
Let $\N$ be a necklace with all the beads fixed. Then a polygon ${P \in \M_{sm}(\N)}$ is a critical point of $\A$ if and only if it is cyclic (i. e. inscribed in a circle). 
\end{oldthm}
\end{samepage}

After describing critical points, the following natural question arises: are these critical points Morse (i. e. whether $\Hess_P \A$, the Hessian of $\A$ at $P$, is a non-degenerate bilinear form on $T_P \M(\N)$) and if they are, what is the Morse index (the dimension of maximal subspace on which $\Hess_P \A$ is negative definite). The state-of-art answer to this question for the case of polygonal linkages requires some more definitions.

\begin{defn}
Let $P$ be a cyclic polygon, $o$ the center of its circumscribed circle, and ${i \in \{1, \ldots, n\}}$.
\begin{itemize}
\item 
{\itshape the central half-angle of the $i$-th side of $P$} is 
\begin{equation} \label{definition of central half-angle ot a side}
\alpha_i(P) = \frac{|\angle p_iop_{i + 1}|}{2} \in [0, \pi/2]
\end{equation}
\item 
{\itshape the orientation of the $i$-th side of $P$} is 
\begin{equation} \label{definition of orientation of a side}
\eps_i(P) = \left\{ \begin{aligned}
1, \text{ if } \angle p_iop_{i + 1} &\in (0, \pi);\\
0, \text{ if } \angle p_iop_{i + 1} &\in \{0, \pi\};\\
-1, \text{ if } \angle p_iop_{i + 1} &\in (-\pi, 0).
\end{aligned} \right.
\end{equation}
\end{itemize}
\end{defn}

\begin{figure}
\label{example of cyclic polygon}
\includegraphics[scale=0.3]{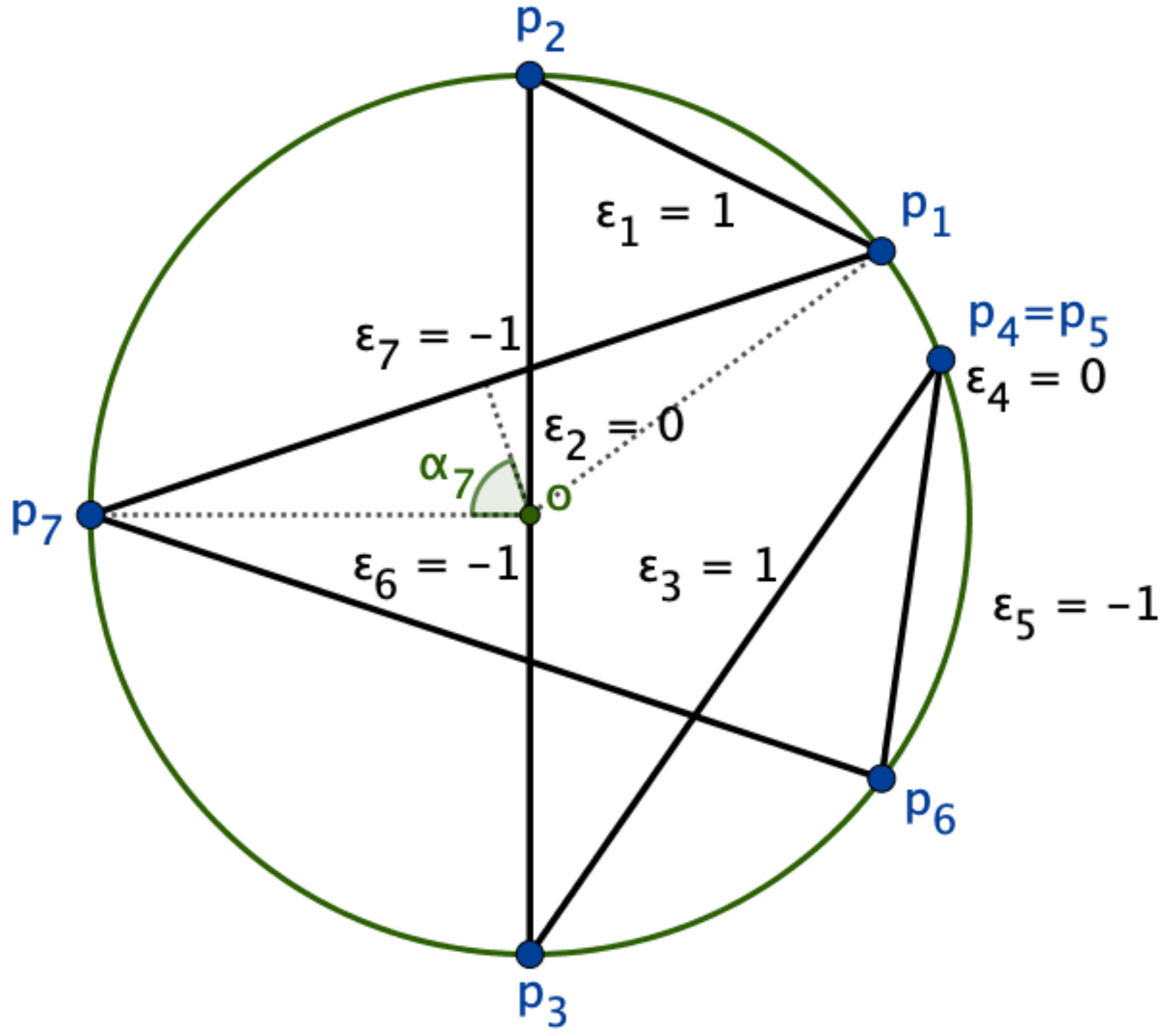}
\caption{A cyclic polygon with some notation}
\end{figure}
'
We will denote by $\C_n$ the {\itshape configuration space of cyclic polygons}, more precisely, 
\begin{equation} \label{definition of the configuration space of cyclic polygons}
\C_n = \left. \left\{P \in \left(\R ^ 2\right) ^ n \left| \begin{aligned}
P \text{ is a cyclic polygon};\\
\AffHull(P) = \R^2
\end{aligned} \right.\right\} \right/\Isom_+(\R^2), 
\end{equation}
For $P \in \C_n$ we will denote by $\Omega_P$ its circumscribed circle, by $o_P$ the centre of $\Omega_P$, and by $R_P$ the radius of $\Omega_P$. 

\begin{defn} \label{definition of an admissible cyclic polygon}
Let $P$ be a cyclic polygon with at least three different vertices. It is called {\itshape admissible} if no edge of $P$ passes through the center of its circumscribed circle. In this case its {\itshape winding number} $w_P = w(P, o)$ with respect to the centre of circumscribed circle is well-defined.
\end{defn}

\begin{oldthm}[{\bfseries Morse indices in the `all beads are fixed' case})\\(Gordon, Khimshiashvili, Panina, Teplitskaya, and Zhukova] \label{Morse index for linkages}~\\ 
Let $\N$ be a necklace without freely moving beads, and let ${P \in \M_{sm}(\N)}$ be an admissible cyclic polygon. Then $P$ is a Morse point of $\A$ if and only if $\sum\limits_{i = 1}^{n} \eps_i\tan \alpha_i \ne 0$ and in this case its Morse index is 
$$
\mu_P(\A) = \#\{i \in \{1, \ldots, n\} \colon \eps_i > 0\} - 1 - 2w_P - \left\{ \begin{aligned}
0, & \text{ if } \sum_{i = 1}^n \eps_i \tan \alpha_i > 0;\\
1, & \text{ otherwise}.
\end{aligned} \right. 
$$
\end{oldthm}
The formula more or less explicitly appeared in \cite{KP12}, \cite{PZ11}, and \cite{Z13}, but in this form, with the precise condition of being Morse, the theorem was proved only in \cite{GPT19}. In view of the theorem, following \cite{GPT19}, we give the following 
\begin{defn} \label{definition of a bifurcating polygon}
An admissible cyclic polygon $P$ is called {\itshape bifurcating} if $\sum\limits_{i = 1}^{n} \eps_i\tan \alpha_i = 0$.
\end{defn}

\subsection{Configuration space of $n$-gons with fixed perimeter}. This is the space $\M((n, L))$ (they are obviously the same for different $L$ so usually $L$ is set to $1$). It is no secret since antiquity that the polygon in $\M((n, L))$ maximising oriented area is the convex regular one. But all the critical points of oriented area together with their indices were determined only in the recent paper \cite{KPS18} by Khimshiashvili, Panina and Siersma. Again, before stating the result, we give
\begin{defn} \label{definition of a regular star and a complete fold}
{\itshape A regular star} is a cyclic polygon $P$ such that all its sides are equal and have the same orientation (see \eqref{definition of orientation of a side}). 

{\itshape A complete fold} is a regular star $P$ with $p_i = p_{i + 2}$ for all $i = 1, \ldots, n$. It exists for even $n$ only.
\end{defn}

\begin{oldthm}[{\bfseries Critical~configurations~and~Morse~indices~in~the~`one~bead~is~fixed'~case})\\(Khimshiashvili, Panina, and Siersma] \label{all about polygons with fixed perimeter}~
\begin{enumerate}
\item $\M((n, L))$ is homeomorphic to $\CP^{n - 2}$.
\item A polygon $P \in \M_{sm}((n, L))$ is a critical point of $\A$, if and only if it is a regular star.
\item The stars with maximal winding numbers are Morse critical points of $\A$.
\item Under assumption that all regular stars are Morse critical points,
the Morse indices are:
$$
\mu_P(\A) = \left\{ \begin{aligned}
&2w_P - 2, &\text{ if } &w_P < 0;\\
&2n - 2w_P - 2, &\text{ if } &w_P > 0;\\
&n - 2, &\text{ if } &P \text{ is a complete fold}.
\end{aligned} \right. 
$$
\end{enumerate}
\end{oldthm}

They also prove a particularly insightful \cite[Lemma 2]{KPS18}:
\begin{oldlem} \label{regular stars are local maxima}
Let $P$ be a regular star which is not a complete fold with $w_P > 0$. Then $P$ is a non-degenerate local maximum on $\C_n$.
\end{oldlem}
In fact, this lemma together with Theorem~\ref{Morse index for linkages} and Lemma~\ref{orthogonality of cyclic and linkages} is enough to omit the assumptions in (4)~of~Theorem~\ref{all about polygons with fixed perimeter}.

\section{Singular locus of the configuration space}
\label{Discussion on singular locus of the configuration space}
\begin{defn} \label{definition of singular and non-singular configurations}
Let $P$ be a configuration of necklace $\N = \big((n_1, L_1), \ldots, (n_k, L_k)\big)$. It is called {\itshape non-singular} if $\L = (\L_1, \ldots, \L_k)$ is a smooth submersion at $P$ (i. e. $\L$ is differentiable at $P$ and its differential $D_P\L \colon T_P\Poly_n \to T_P\R^k$ is a surjective linear map), otherwise it is called {\itshape singular}. 
\end{defn}

First we give a geometric characterisation of singular configurations. Consider a polygon ${P = (p_1, \ldots, p_n)}$, with $p_i = (x_i, y_i) \in \R^2$ and $l_i = |p_{i + 1} - p_i| \ne 0$ for all $i = 1, \ldots, n$ and define $\beta_i$ to be the oriented angle between vectors $(1, 0)$ and ${(p_{i + 1} - p_i)}$. We also denote by $s(j) = n_1 + \ldots + n_{j - 1} + 1$ the index of the $j$-th fixed bead. Then $\L_j$ are differentiable at $P$ and the derivatives of $\L_j$ with respect to $x_i$ and $y_i$ look as follows:

\begin{align}
\frac{\partial \L_j}{\partial x_i}(P) &= \left\{ \begin{aligned}
&- \cos \beta_i,  &\text{if} \q &i = s(j);\\
&\cos \beta_{i - 1} - \cos \beta_i, &\text{if} \q &i \in j^*\setminus\{s(j)\};\\
&\cos \beta_{i - 1},  &\text{if} \q &i = s(j + 1);\\
&0,  &\text{otherwise.} &
\end{aligned}
\right. \label{partial derivative of L wrt x} \\
\frac{\partial \L_j}{\partial y_i}(P) &= \left\{ \begin{aligned}
&- \sin \beta_i,  &\text{if} \q &i = s(j);\\
&\sin \beta_{i - 1} - \sin \beta_i, &\text{if} \q &i \in j^* \setminus\{s(j)\};\\
&\sin \beta_{i - 1},  &\text{if} \q &i = s(j + 1);\\
&0,  &\text{otherwise.} &
\end{aligned}
\right.  \label{partial derivative of L wrt y}
\end{align}

The indices of the form $s(j)$ for $j = 1, \ldots, k$ will be called {\itshape boundary} and all other indices will be called {\itshape inner}. We now establish a criterion for a configuration of a necklace to be singular.

\begin{lem} \label{criterion for being singular configuration}
A configuration $P \in \Poly_n$ of the necklace $\N = \big((n_1, L_1), \ldots, (n_k, L_k)\big)$ is singular if and only if one of the following holds:
\begin{enumerate}
\item $l_i = 0$ for some $i = 1, \ldots, n$;
\item $P$ fits in a straight line in such a way that $\beta_i = \beta_{i - 1}$ for all inner indices $i$.
\end{enumerate}
\end{lem}
\begin{proof}
The first condition is equivalent to $\L$ being differentiable at $P$. Therefore, what is left to prove, is that for $P \in \tilde \M(\N)$ with no vanishing sides, the second condition hold if and only if the gradients $\grad_P \L_1, \ldots, \grad_P \L_k$ are linearly dependent.   

Suppose that $\lambda_1\grad_P \L_1 + \ldots + \lambda_k \grad_P \L_k = 0$ is a non-trivial vanishing linear combination. If $\lambda_j \ne 0$, then, using formulae \eqref{partial derivative of L wrt x} and \eqref{partial derivative of L wrt y} for boundary index $s(j)$, we get $-\lambda_j \cos \beta_{s(j)} + \lambda_{j - 1} \cos \beta_{s(j) - 1} = 0$ and $-\lambda_j \sin \beta_{s(j)} + \lambda_{j - 1} \sin \beta_{s(j) - 1} = 0$. It means that points $\lambda_j(\cos \beta_{s(j)}, \sin \beta_{s(j)}) \ne (0, 0)$ and $\lambda_{j - 1} (\cos \beta_{s(j)-1}, \sin \beta_{s(j)-1})$ coincide, which implies that $2(\beta_{s(j)} - \beta_{s(j) - 1}) = 0$ and ${\lambda_{j - 1} = \cos(\beta_{s(j)} - \beta_{s(j) - 1})\lambda_{j}} \ne 0$. It follows then that $\lambda_j \ne 0$ for all $j = 1, \ldots, k$, consequently, (we now use \eqref{partial derivative of L wrt x} and \eqref{partial derivative of L wrt y} for inner indices) $\beta_i = \beta_{i - 1}$ for all inner indices $i$, meaning that $P$ is composed of straight segments of lengths $L_1, \ldots, L_k$. Taking in account previously deduced formula $2(\beta_{i} - \beta_{i - 1} ) = 0$ for boundary $i$, we get that $P$ does satisfies condition (2). Reversing the above argument, we get the reverse implication.
\end{proof}

Now let $\tilde{\M_{sm}}(\N)$ be the set of non-singular configurations of necklace $\N$ and $\M_{sm}(\N)$ be the non-singular part of $\M(\N)$:
\begin{equation}
\M_{sm}\big(\N) = \frac{\tilde{\M_{sm}}(\N)}{\Isom_+(\R^2)} = \left.\left\{P \in \Poly_n \left| \begin{aligned}
P \text{ is a non-singular}\\
\text{configuration of } \N
\end{aligned} \right.\right\}\right/\Isom_+(\R^2)
\end{equation}

If these spaces are non-empty, they are smooth manifolds, which generalises previous results on smoothness of configuration spaces of polygonal linkages by Kapovich---Millson~\cite{KM95} and Farber~\cite{F08}. To state the precise result, we need the following 
\begin{defn} \label{definition of a realisable necklace}
A necklace $\N = ((n_1, L_1), \ldots, (n_k, L_k)\big)$ is called {\itshape realisable} if for all $j = 1, \ldots, k$, such that $n_j = 1$, inequality $2L_j < L_1 + \ldots + L_k$ holds.
\end{defn}

\begin{prop} \label{smooth piece of configuration space of necklace}
Let $\N$ be a realisable necklace. Then
\begin{enumerate}
\item $\tilde{\M_{sm}}(\N)$ is a smooth submanifold of dimension $2n - k$ in $\Poly_n = \R^{2n}$;
\item $\M_{sm}(\N)$ is a topological manifold of dimension $2n - k - 3$ with a unique smooth structure making the quotient map $\tilde{\M_{sm}}(\N) \to \M_{sm}(\N)$ a smooth submersion;
\item the oriented area function $\A$ is a smooth function on $\M_{sm}(\N)$.
\end{enumerate}
\end{prop} 
\begin{proof}
It follows from Lemma~\ref{criterion for being singular configuration}, that the inequalities $2L_j < L_1 + \ldots + L_k$ are necessary and sufficient for $\tilde{\M_{sm}}(\N)$ to be non-empty. 

The first claim is clear since $\tilde{\M_{sm}}(\N)$ is locally a level of a smooth submersion ${\L = (\L_1, \ldots, \L_k)\colon \left(\R^2\right) ^ n \to \R^k}$.

To establish the second claim, we first note that $\M_{sm}(\N)$ is an orbit space of the action of 3-dimensional Lie group $\Isom_+(\R^2)$ on smooth manifold $\tilde{\M_{sm}}(\N)$. Thus, it suffices to observe that the action is free and proper, which is indeed the case.

The third claim is obvious since the smooth structure on $\M_{sm}(\N)$ is induced from $\Poly_n$ and oriented area $\A$ (see Definition~\ref{definition of oriented area}) is a smooth function on $\Poly_n$ preserved by the action of $\Isom_+(\R^2)$.
\end{proof}

\section{Main results: critical configurations and their Morse indices \\in the general case} 
\label{Main results} 
The first theorem describes critical points of oriented area on configuration spaces of necklaces, generalising Theorem~\ref{critical points for linkages} and (2)~in~Theorem~\ref{all about polygons with fixed perimeter}. Recall that $j^*$ is the set of indices corresponding to the $j$-th piece of a necklace (Definition~\ref{definition of j^*}) and $\eps_i(P)$ is an orientation of the $i$-th side of a cyclic polygon $P$ (Definition~\ref{definition of orientation of a side}).
\begin{thm}[{\bfseries Critical configurations in the general case}] \label{critical points}~\\
A polygon $P \in \M_{sm}\big((n_1, L_1), \ldots, (n_k, L_k)\big)$ is a critical point of $\A$ if and only if all of the following conditions hold:
\begin{enumerate}
\item $P$ is cyclic;
\item $\forall i \in j^*\colon l_i(P) = L_j/n_j$;
\item $\forall i_1, i_2 \in j^*\colon \eps_{i_1}(P) = \eps_{i_2}(P)$.
\end{enumerate}
\end{thm}

Before proving the theorem, let us make two remarks.
\begin{rem}
The statement of the theorem in plane English goes as follow: {\itshape a non-singular configuration of a necklace is a critical point of oriented area if and only if all the beads lie on some circle in such a way that the arc between every two consecutive fixed beads is evenly divided by the moving beads between them.} 
\end{rem}
\begin{rem}
The proof essentially is a reformulation of geometric arguments into the language of Lagrange multipliers, so we first write partial derivatives of $\A$ with respect to $x_i$ and $y_i$:
\begin{align}
2 \cdot \frac{\partial \A}{\partial x_i}(P) &= l_{i - 1}\sin \beta_{i - 1} + l_i\sin \beta_i \label{partial derivative of A wrt x} \\
2 \cdot \frac{\partial \A}{\partial y_i}(P) &= -l_{i - 1}\cos \beta_{i - 1} - l_i\cos \beta_{i} \label{partial derivative of A wrt y}
\end{align}
We follow the convention ${0 \cdot \text{undefined} = 0}$ hence both sides are defined for all $P \in \Poly_n$. 
\end{rem}

\begin{proof}[Proof of Theorem~\ref{critical points}]
Let $P$ be a non-singular configuration of necklace ${\N = \big((n_1, L_1), \ldots, (n_k, L_k)\big)}$. Then $P$ is a critical point of $\A$ if and only if there exist $\lambda_1, \ldots, \lambda_k \in \R$, such that ${2\grad_P \A = \lambda_1 \grad_P \L_1 + \ldots + \lambda_k \grad_P \L_k}$. 

Assume that $2\cdot \grad_P \A = \lambda_1 \grad_P \L_1 + \ldots + \lambda_k \grad_P \L_k$. Then for inner index $i$ corresponding to the $j$-th piece of $\N$ we have
\begin{align*}
l_{i - 1} \sin \beta_{i - 1} + l_i \sin \beta_i = \lambda_j\left(\cos \beta_{i - 1} - \cos \beta_i \right);\\
-l_{i - 1} \cos \beta_{i - 1} - l_i \cos \beta_i = \lambda_j\left(\sin \beta_{i - 1} - \sin \beta_i \right).
\end{align*}
If $\beta_i = \beta_{i - 1}$, then $l_i = l_{i - 1} = 0$, but $P$ is non-singular, so it cannot be the case by Lemma~\ref{criterion for being singular configuration}. The only other possibility for this equations to hold is $l_{i - 1} = l_i$ and $\lambda_j = l_i\cot\left(\frac{\beta_i - \beta_{i - 1}}{2}\right)$. Since we have such equations for all inner indices corresponding to $j$, we get that $\forall i_1, i_2 \in j^*\colon l_{i_1} = l_{i_2}$, which implies condition (2) of the theorem. Moreover, for all $i \in j^*\setminus s(j)$ we get $\cot\left(\frac{\beta_i - \beta_{i - 1}}{2}\right) = \frac{n_j\lambda_j}{L_j}$, therefore $\beta_i - \beta_{i - 1}$ is the same for all $i \in j^*\setminus s(j)$, which implies that there is a circle $\Omega_j$ with centre $o_j$ such that conditions (2) and (3) of the theorem hold. It now remains to prove that $P$ is cyclic, i. e. $o_j$ is the same for all $j = 1, \ldots, k$. If $P$ is a smooth point of $\M_{sm}((1, l_1), \ldots, (1, l_n)) \subset \M_{sm}\big(\N\big)$, in other words, if $P$ does not fit in a straight line, then we are done by Theorem~\ref{critical points for linkages}. Suppose that $P$ fits in a straight line. Pick a boundary vertex $i = s(j + 1)$ and denote $l^j = L_j/n_j$. We have
\begin{align*}
l^j \sin \beta_{i - 1} + l^{j + 1}\sin \beta_i = \lambda_j\cos \beta_{i - 1} - \lambda_{j + 1}\cos \beta_i;\\
-l^j \cos \beta_{i - 1} - l^{j + 1}\cos \beta_i = \lambda_j\sin \beta_{i - 1} - \lambda_{j + 1}\sin \beta_i.
\end{align*}
Since $P$ fits in a straight line, $2(\beta_i - \beta_{i - 1}) = 0$. If $\beta_{i - 1} = \beta_i = \beta$, then the points $(l^j + l^{j + 1})(\cos \beta, \sin \beta)$ and $(\lambda_j - \lambda_{j + 1})(\cos (\beta + \pi/2), \sin (\beta + \pi/2))$ coincide which cannot be the case since $l^j, l^{j + 1} > 0$. If $\beta_{i-1} = \beta_i + \pi = \beta+\pi$, then the points $(l^j - l^{j + 1})(\cos \beta, \sin \beta)$ and $(\lambda_j + \lambda_{j + 1})(\cos (\beta + \pi/2), \sin (\beta + \pi/2))$ coincide, which implies that $l^j = l^{j + 1}$. Since this is the case for all $j$, $P$ is a complete fold and thus indeed is cyclic. 

Now assume that a non-singular configuration $P$ of necklace $\N$ satisfies conditions (1)--(3). Let $\Omega$ be its circumscribed circle with centre $o$. Denote by $\gamma_j$ the oriented angle $\angle p_{s(j)op_{s(j) + 1}}$ and set $\lambda_j = l_i\cot\left(\gamma_j/2\right)$ for some index $i$ corresponding to $j$. Since $\gamma_j = \beta_i - \beta_{i - 1}$ for inner indices $i$,  equality $2\cdot \grad_P \A = \lambda_1 \grad_P \L_1 + \ldots + \lambda_k \grad_P \L_k$ holds in all inner indices. For a boundary index $i = s(j + 1)$ we can (performing rotation around $o$) assume, that $\beta_{i - 1} = 0$, and what we need to check then is the following two equalities:
\begin{align*}
l^{j + 1} \sin \beta_i &= l^j \cot(\gamma_{i - 1}/2) - l ^{j + 1} \cot (\gamma_i/2) \cos 
\beta_i;\\
-l ^ j - l ^{j + 1}\cos \beta_i &= -l^{j + 1} \cot(\gamma_i/2) \sin \beta_i,
\end{align*}
Putting the origin at $o$, we note that
$$
p_{i + 1} - p_i = l^{j + 1} \cdot (\cos \beta_i, \sin \beta_i), \q p_{i + 1} + p_i = l^{j + 1}\cot\frac{\gamma_i}{2} \cdot (-\sin \beta_i, \cos \beta_i), \q p_i = \left(\frac{l_j}{2}, -\frac{l_j}{2}\cot\frac{\gamma_{i - 1}}{2}\right),
$$
and thus the desired equalities are just the coordinate manifestations of the obvious identity
$$
\frac{p_{i + 1} - p_i}{2} - \frac{p_{i + 1} + p_i}{2} + p_i = (0, 0)
$$
\end{proof}

The following theorem provides a criterion for an admissible cyclic polygon to be a Morse point of oriented area and gives a formula for its Morse index. It generalises Theorem~\ref{Morse index for linkages} and allows one to omit the assumptions in (4)~of~Theorem~\ref{all about polygons with fixed perimeter}. 
\begin{thm}[{\bfseries Morse indices in the general case}] \label{Morse index}~\\
Let $\N = \big((n_1, L_1), \ldots, (n_k, L_k)\big)$ be a realisable necklace (see Definition~\ref{definition of a realisable necklace}), and $P \in \M_{sm}(\N)$ be an admissible (see Definition~\ref{definition of an admissible cyclic polygon}) critical point of oriented area $\A$.
Then $P$ is a Morse point of $\A$ if and only if it is not a bifurcating polygon (see Definition~\ref{definition of a bifurcating polygon}). In this case its Morse index can be computed as follows:
$$
\mu_P(\A) = \frac{1}{2}\sum_{j = 1}^k (2n_j - 1) \cdot (E_j + 1) - 1 - 2w_P - \left\{ \begin{aligned}
&0 \text{ if } \sum_{j = 1}^k n_jE_j\tan A_j > 0;\\
&1 \text{ otherwise,}
\end{aligned} \right.
$$ 
where $E_j = \eps_i$ and $A_j = \alpha_i$ for some $i \in j^*$ (due to Theorem~\ref{critical points} this does not depend on the choice of i).
\end{thm}
\begin{proof}
Let $P$ be as in the theorem. First, we can split the tangent space of $\M_{sm}(\N)$ at the critical point $P$ into subspaces that are orthogonal with respect to the Hessian form $\Hess_P \A$. For this, given a polygon $P$, we introduce the following submanifolds in $\M_{sm}(\N)$: 
\begin{enumerate}
\item $\E^P = \M_{sm}((1, l_1), \ldots, (n, l_n)) \subset \M_{sm}(\N)$ --- the space of all polygons having the same edge length as $P$;
\item $\C^P = \M_{sm}(\N) \cap \C$ --- the subspace of cyclic polygons;
\item 
$
\C_j^P = \left\{Q \in \M_{sm}(\N) \left| \begin{aligned} 
(q_{s(j)}, \ldots, q_{s(j + 1)}) \text{ is cyclic}\\
q_i = p_i \text{ for } i \notin j^* \setminus \{s(j)\}\\
\end{aligned} \right.\right\}
$
for $j = 1, \ldots, k$.
\end{enumerate}

We will deduce the theorem from the Lemmata \ref{orthogonality of cyclic and linkages}, \ref{orthogonality of cyclic pieces}, and \ref{Morse index for cyclic} (see Sections~\ref{Orthogonality with respect to the Hessian of oriented area}~and~\ref{Addenda} for their proofs).
\begin{lem} \label{orthogonality of cyclic and linkages}
Let $P$ be as in Theorem~\ref{Morse index}. Then
\begin{enumerate}
\item $\E^P$ around $P$ is a smooth submanifold in $\M_{sm}$ of dimension $n - 3$;
\item $\C^P$ around $P$ is a smooth submanifold in $\M_{sm}$ of dimension $n - k$;
\item $\E^P$ and $\C^P$ intersect transversally at $P$, i. e.
$
T_P \M_{sm} = T_P\E^P \oplus T_P\C^P;
$
\item $T_P\E^P$ and $T_P\C^P$ are orthogonal with respect to bilinear form $\Hess_P \A$.
\end{enumerate}
\end{lem}

One can note that none of $\C_j^P$ is contained in $\C^P$. Nonetheless, from the following lemma one sees that in the first approximation they very much are.
\begin{lem} \label{orthogonality of cyclic pieces}
Let $P$ be as in Theorem~\ref{Morse index}. Then
\begin{enumerate}
\item $\C^P_j$ around $P$ is a smooth submanifold in $\M_{sm}$ of dimension $n_j - 1$;
\item 
$$
T_P\C^P = \bigoplus_{j = 1}^k T_P\C^P_j
$$
\item $T_P\C_j^P$ are pairwise orthogonal with respect to bilinear form $\Hess_P \A$.
\end{enumerate}
\end{lem}

It remains to compute Morse index of $P$ with respect to $\A$ on each of $\C_j^P$.
\begin{lem} \label{Morse index for cyclic}
Suppose that $P \in \C_{n+1}$ is such that $l_1 = \ldots = l_n = L/n$ and $\eps_i = 1$ ($\eps_i = -1$) for $i = 1, \ldots, n$. Then $P$ is a non-degenerate local maximum (minimum) of oriented area on ${\M_{sm}((n, L), (1, l_n)) \cap \C_{n+1}}$.
\end{lem}

Now we are ready to prove the theorem. From Lemmata~\ref{orthogonality of cyclic and linkages}~and~\ref{orthogonality of cyclic pieces}, $P$ is a Morse point of $\A$ on $\M_{sm}$ if and only if it is a Morse point of $\A$ on $\E^P$ and all of $\C^P_j$. Since $P$ is always a Morse point on each of $\C^P_j$ (because by Lemma~\ref{Morse index for cyclic} it is a non-degenerate local extremum), it is a Morse point of $\A$ on $\M_{sm}$ if and only if it is a Morse point of $\A$ on $\E^P$, which is equivalent to $P$ not being bifurcating by Theorem~\ref{Morse index for linkages}.

Moreover, again using Lemmata~\ref{orthogonality of cyclic and linkages}~and~\ref{orthogonality of cyclic pieces}, we conclude that if $P$ is a Morse point of $\A$ on $\M_{sm}$, then its Morse index is
$$
\mu_P^{n_1, \ldots, n_k}(\A) = \mu_P^{\E^P}(\A) + \mu_P^{\C^P}(\A) = \mu_P^{1, \ldots, 1} (\A) + \sum_{j = 1}^k \mu_P^{\C^P_j}(\A).
$$

From Theorem~\ref{Morse index for linkages} we know that
$$
\mu_P^{1, \ldots, 1}(\A) = 
\frac{1}{2}\sum_{j = 1}^k n_j(E_j + 1) - 1 - 2\omega - \left\{ \begin{aligned}
0, & \text{ if } \sum_{j = 1}^k n_jE_j \tan A_j > 0;\\
1, & \text{ otherwise}.
\end{aligned} \right. 
$$
From Lemma~\ref{Morse index for cyclic} and (1)~of~Lemma~\ref{orthogonality of cyclic pieces} we get
$$
\mu_P^{\C_j^P}(\A) = \frac{1}{2}(n_j - 1) \cdot (E_j + 1).
$$
Summing all up, we obtain the desired formula.
\end{proof}

\section{Orthogonality with respect to the Hessian of oriented area}
\label{Orthogonality with respect to the Hessian of oriented area}
Let us remind that $\C_n$ is the configuration space of cyclic polygons with at least three different vertices (see \eqref{definition of the configuration space of cyclic polygons}).
First, we parametrise $\C_n$ smoothly. For this we introduce
$$
\H_n = \left.\left\{(\theta_1, \ldots, \theta_n) \in \left(S^1\right) ^ n \mid \text{there are at least three different points among } \theta_1, \ldots, \theta_n\right\}\right/S^1,
$$
where $S^1$ acts on $\left(S^1\right)^n$ diagonally by rotations. Consider the following map
\begin{equation}
\begin{aligned}
\tilde \ph &\colon \left(\left(S^1\right)^n \setminus Diag\right) \times \R_{>0} &\to &\left(\R^2\right)^n \setminus Diag\\
\tilde \ph &\colon (\theta_1, \ldots, \theta_n, R) &\mapsto &R\cdot \left((\cos \theta_1, \sin \theta_1), \ldots, (\cos \theta_n, \sin \theta_n)\right),
\end{aligned}
\end{equation}
\begin{lem}\label{parametrisation of cyclic polygons}
$\tilde \ph$ induces a diffeomorphism $\ph\colon \H_n \times \R_{>0} \to \C_n$.
\end{lem}
\begin{proof}
$\ph$ is obviously a bijection, so the only thing we need to check is that the Jacobian of $\tilde \ph$ has rank $(n + 1)$ at every point. In fact, it is just a statement of the form `$S^1 \times \R_{>0}$ is diffeomorphic to $\R^2 \setminus \{0\}$ via polar coordinates', but we compute the Jacobian for the sake of completeness:
$$
\Jac \ph = \begin{pmatrix}
\Jac_1 \ph\\
\vdots\\
\Jac_n \ph\\
\Jac_{n + 1} \ph
\end{pmatrix} 
=\begin{pmatrix}
-R\sin \theta_1 & R \cos \theta_1 & 0 & \ldots & 0 & 0\\
\vdots & \vdots & \vdots & \ddots & \vdots & \vdots\\
0 & 0 & 0 & \ldots & -R\sin \theta_n & R \cos \theta_n\\
\cos \theta_1 & \sin \theta_1 & \cos \theta_2 & \ldots & \cos \theta_n & \sin \theta_n
\end{pmatrix}
$$
The first $n$ rows are obviously linearly independent. Suppose one has ${\Jac_{n + 1} \ph = \lambda_1 \Jac_1 \ph + \ldots + \lambda_n \Jac_n \ph}$. Then for any $i = 1, \ldots, n$ one gets 
$$
(\cos \theta_i, \sin \theta_i) = \lambda_1(-R\sin \theta_i, R\cos \theta_i) = \lambda_1 R \left(\cos \left(\theta_i + \pi/2\right), \cos \left(\theta_i + \pi/2\right)\right),
$$
which implies $\lambda_i = 0$, a contradiction.
\end{proof}

Now we provide local coordinates for $\C_n$.
\begin{lem} \label{side length coordinates for cyclic polygons}
Let $P \in \C_n$ be an admissible non-bifurcating cyclic polygon with side lengths $l_1, \ldots, l_n>0$. For $Q \in \C_n$ let $t_i(Q) = l_i(Q) - l_i$. Then $(t_1, \ldots, t_n)$ are smooth local coordinates for $\C_n$ around $P$.
\end{lem}
\begin{proof}
In view of Lemma~\ref{parametrisation of cyclic polygons} we just need to show that for
$$
\psi \colon \H_n \times \R_{>0} \to \R^n, \q (\theta_1, \ldots, \theta_n, R) \mapsto R\cdot(\sqrt{2 - 2\cos (\theta_2 - \theta_1)}, \ldots, \sqrt{2 - 2\cos (\theta_1 - \theta_n)})
$$
$\Jac \psi$ is of rank $n$ at points where $\theta_1 \ne \theta_2 \ne \ldots \ne \theta_n \ne \theta_1$. Indeed, $\Jac \psi$ is
$$
\begin{pmatrix}
R\frac{\sin(\theta_1 - \theta_2)}{\sqrt{2 - 2 \cos(\theta_1 - \theta_2)}}  & 0 & \ldots & 0 & R\frac{\sin(\theta_1 - \theta_n)}{\sqrt{2 - 2 \cos(\theta_1 - \theta_n)}}\\
R\frac{\sin(\theta_2 - \theta_1)}{\sqrt{2 - 2 \cos(\theta_2 - \theta_1)}}  & R\frac{\sin(\theta_2 - \theta_3)}{\sqrt{2 - 2 \cos(\theta_2 - \theta_3)}} & \ldots & 0 & 0\\
\vdots & \vdots & \vdots & \ddots & \vdots\\
0 & 0 & \ldots & R\frac{\sin(\theta_{n - 1} - \theta_n)}{\sqrt{2 - 2 \cos(\theta_{n - 1} - \theta_n)}} & 0\\
0 & 0 & \ldots & R\frac{\sin(\theta_n - \theta_{n - 1})}{\sqrt{2 - 2 \cos(\theta_n - \theta_{n - 1})}} & R\frac{\sin(\theta_n - \theta_1)}{\sqrt{2 - 2 \cos(\theta_n - \theta_1)}}\\
\sqrt{2 - 2\cos (\theta_2 - \theta_1)} & \sqrt{2 - 2\cos (\theta_3 - \theta_2)} & \ldots & \sqrt{2 - 2\cos (\theta_{n} - \theta_{n - 1})} & \sqrt{2 - 2\cos (\theta_1 - \theta_n)}
\end{pmatrix}
$$
Since $2(\theta_{i + 1} - \theta_{i}) \ne 0$, all the entries are defined and non-zero. Consider a vanishing non-trivial linear combination of columns. The form of first $n$ rows forces the coefficient at the $i$-th column to be equal (up to the common multiplier) to $\frac{\sqrt{2 - 2 \cos(\theta_i - \theta_{i+1})}}{\sin(\theta_i - \theta_{i + 1})}$, but then for the last row we have
$$
0 = \sum_{i = 1}^n\frac{2 - 2\cos (\theta_i - \theta_{i+1})}{\sin(\theta_i - \theta_{i + 1})} = 2\sum_{i = 1}^n \tan \left(\frac{\theta_i - \theta_{i + 1}}{2}\right),
$$
which means exactly that $P$ is bifurcating and contradicts the assumptions of the lemma. Thus, $\Jac \psi$ has rank $n$ as desired.
\end{proof}

First we prove orthogonality of $T_P\E^P$ and $T_P\C^P$ in $T_P\M_{sm}(\N)$.
\begin{proof}[Proof of Lemma~\ref{orthogonality of cyclic and linkages}]
To prove the first to claims let us note that smooth structures on $\E^P$, $\C^P$, and $\M_{sm}(N)$ come from the smooth structure on $\Poly_n = \left(\R ^ 2\right) ^ n$. Thus, the first claim immediately follows from Lemma~\ref{criterion for being singular configuration}, as the only cyclic polygon fitting into a straight line is a complete fold, which is not admissible. The dimension of $\E^P$ is computed according to (2) in Proposition~\ref{smooth piece of configuration space of necklace}. From Lemma~\ref{parametrisation of cyclic polygons} it follows that $\C_n$ around $P$ is a smooth submanifold in $\left.\Poly_n\right/ \Isom_+$, and from Lemma~\ref{side length coordinates for cyclic polygons} we deduce that $\C^P$ around $P$ is a smooth $(n - k)$-dimensional submanifold of $\C_n$ as it is a preimage of the linear subspace of codimension $k$ in $\R^n$ under the map $Q \mapsto (t_1(Q), \ldots, t_n(Q))$. Thus the second claim is also proved.

The third claim is equivalent (by dimension count) to representability of every vector in $T_P\M_{sm}(\N)$ as a sum of two vectors from $T_P\E^P$ and $T_P\C^P$ respectively, but this is indeed the case since every polygon $Q$ near $P$ in $\M_{sm}(\N)$ can be obtained by first a move in $\C^P$ making the sides of desired length (by Lemma~\ref{side length coordinates for cyclic polygons}) and then by a move inside $\E^Q$.

Finally, we establish the forth claim. Consider $v \in T_PC$ and $w \in T_PE$. To compute $\Hess_P \A (v, w)$, we choose a curve $\gamma\colon (-\eps, \eps) \to T_PC$ such that $\gamma(0) = P$ and $\gamma'(0) = v$, then we extend $w$ to a vector field $W(t) \in T_{\gamma(t)}E^{\gamma(t)}$ along $\gamma$. Then
$$
\Hess_P \A (v, w) =  \left. \frac{d}{dt}\right|_{t = 0} d_{\gamma(t)}\A(W(t)).
$$
But $d_{\gamma(t)}\A$ vanishes on $T_{\gamma(t)}E^{\gamma(t)}$ by Theorem~\ref{critical points}
\end{proof}

To split $T_P\C^P$ further, we need the following 
\begin{lem} \label{derivatives of radius and diagonals}
Let $P \in \C_n$ be an admissible non-bifurcating cyclic polygon such that $l_1 = l_2$ and $\angle p_1op_2 = \angle p_2op_3$, where $o$ is the centre of the circumscribed circle $\Omega$. Let $V$ be a local vector field around $P$ equal to $\left( \frac{\partial}{\partial t_1} - \frac{\partial}{\partial t_2}\right)$ in the coordinates from Lemma~\ref{side length coordinates for cyclic polygons}. Then $VR(P) = 0$ and $Vd_{ab}(P) = 0$ for $a, b \in \{1, \ldots, n\} \setminus \{2\}$, where $Vf$ if the derivative along $V$ of function $f$, $R(Q)$ is the radius of circumscribed circle of $Q$ and $d_{ab}(Q) = |q_b - q_a|$.
\end{lem}
\begin{proof}
Consider a curve $P(s)\colon (-\eps, \eps) \to \C_n, \q (t_1, \ldots, t_n)(P(s)) = (s, -s, 0, \ldots, 0)$. We choose representatives $\tilde P(s) \in \Poly_n$ in such a way that $o_{\tilde P(s)} = (0, 0)$ and $(p_3 - p_1)$ is codirectional with $x$-axes. Notice that $\tilde P(-s)$ is obtained from $\tilde P(s)$ by the following procedure: ${p_i(-s) = p_i(s)}$ for $i \ne 2$ and $p_2(-s)$ is symmetric to $p_2(s)$ relative to $y$-axes. From this it follows that ${\tilde P(s) - \tilde P(-s)) = (0, 0, 2\eta, 0, \ldots, 0)}$ for some $\eta > 0$. Hence all $p_i$ for $i \ne 2$ are not moving in the first approximation, which implies the statement of the lemma.
\end{proof}

This lemma allows us to relate $\C^P_j$ with $\C^P$ and thus prove Lemma~\ref{orthogonality of cyclic pieces}.
\begin{proof}[Proof of Lemma~\ref{orthogonality of cyclic pieces}]
The space 
$
\left\{Q \in \M_{sm}(\N) \left| \begin{aligned} 
q_i = p_i \text{ for } i \notin j^* \setminus \{s(j)\}
\end{aligned} \right.\right\}
$
is a smooth submanifold in $\M_{sm}(\N)$ diffeomorphic to $\M_{sm}((n_j, L_j), (1, |p_{s(j + 1)} - p_{s(j)}|))$. Under this identification, $\C^P_j$ is just a $\C^P$. Applying (2) of Lemma~\ref{orthogonality of cyclic and linkages} to $\M_{sm}((n_j, L_j), (1, |p_{s(j + 1)} - p_{s(j)}|))$, we get the first claim.

To establish the second claim we first prove that $T_P\C_j^P \le T_P\C^P$. Indeed, consider the coordinates from Lemma~\ref{side length coordinates for cyclic polygons}. On the one hand, when we consider cyclic polygons coordinatised by $(t_1, \ldots, t_n)$, the vectors $\left(\frac{\partial}{\partial t_{i - 1}} - \frac{\partial}{\partial t_i}\right)$ for inner $i$ form a basis of $T_P\C^P$. On the other hand, when we consider $\C_j^P$ coordinatised by $(s_i)_{i \in j^*\setminus \{s(j)\}}$, where $s_i = l_i(Q) - l_i(P)$, the vectors $\left(\frac{\partial}{\partial s_{i - 1}} - \frac{\partial}{\partial s_i}\right)$ for ${i \in j^* \setminus \{s(j)\}}$ form a basis of $T_P\C_j^P$. But by Lemma~\ref{derivatives of radius and diagonals}, this tangent vectors are the same, so the claim is proven. In fact, we proved not only that $T_P\C_j^P \le T_P\C^P$, but also that $T_P\C^P = \bigoplus_{j = 1}^k T_P\C_j^P$, since aforementioned basis of $T_P\C^P$ is a disjoint union of bases of $T_P\C_j^P$.

Now we pass to proving the third claim. Consider $v \in T_P\C_j^P$ and $w \in T_P\C_h^P$, take a curve $\gamma\colon (-\eps, \eps) \to \C_j^P$ such that $\gamma(0) = P$ and $\gamma'(0) = v$, and a curve $\sigma\colon (-\eps, \eps) \to \C_h^P$, such that $\sigma(0) = P$ and $\sigma'(0) = w$. Then extend $w$ to a vector field ${W(t) \in T_{\gamma(t)}\M_{sm}(\N)}$ along $\gamma$ by setting $W(t) = \sigma_t'(0)$, where $\sigma_t\colon (-\eps, \eps) \to \C_h^{\gamma(t)}$ is such that $\sigma_t(0) = \gamma(t)$ and for all $i \in j^* \setminus {s(j)}$ the $i$-th vertex of $\sigma_t(s)$ is the same as the $i-th$ vertex of $\sigma(s)$. Then
$$
\Hess_P \A (v, w) =  \left. \frac{d}{dt}\right|_{t = 0} W(t)\A,
$$
and it vanishes since $W(t)\A$ does not depend on $t$.
\end{proof}

\section{Configuration spaces of polygons with perimeter and one side length fixed}
\label{Addenda}
These are the spaces $\M\big((n, L), (1, l)\big)$ for $L \ge l$. The other name for such a space, {\itshape the space of broken lines of fixed length with fixed endpoints}, comes from the canonical choice of representative of each orbit: the first vertex has coordinates $(0, 0)$ and the last one --- $(l, 0)$. Our interest in these spaces was first motivated by the fact that they are simple enough to be studied completely, but then it turned out that they are important for understanding the case of a general necklace.
\begin{prop}[{\bfseries Configuration~space~in~the~`two~consecutive~beads~are~fixed'~case}]~\label{topological type in the two beads case}\\
Let $L > l$ and $n \ge 2$. Then $\M\big((n, L), (1, l)\big)$ is homeomorphic to the sphere $S^{2n - 3}$.
\end{prop}
\begin{proof}
By setting $p_1 = (0, 0)$ and $p_{n + 1} = (l, 0)$ we identify $\M\big((n, L), (1, l)\big)$ with the level set
\begin{align*}
F^{-1}(L) = \left\{(p_2, \ldots, p_n) \in \left(\R^2\right)^{n - 1} \left| F(p_2, \ldots, p_n) = L\right.\right\}, \text{ where}\\
{F(p_2, \ldots, p_n) = |p_2| + |p_3 - p_2| + \ldots + |p_n - p_{n - 1}| + |(l, 0) - p_n|}.
\end{align*}
$F$ is a convex function as sum of convex functions. The sublevel set $F^{-1}((\infty, 0])$ is bounded since if any of $|p_i|$ is greater than $\R$, then $F(p_2, \ldots, p_n) \ge \R$ by triangle inequality. Also, the set $F^{-1}((\infty, 0))$ is non-empty, since if all of the $p_i$ are in the disk of radius $\delta$ around $(l/2, 0)$, then $F(p_2, \ldots, p_n) < (l/2 + \delta) + (n - 3) \delta + (l/2 + \delta) = l + (n - 1)\delta$, which is less than $L$ for small $\delta$. So, $F^{-1}(L)$ is a boundary of the compact convex set $F^{-1}((\infty, 0]) \subset \left(\R^2\right)^{n - 1} $ with non-empty interior and thus is homeomorphic to $S^{2n - 3}$.
\end{proof}

As a special case of Theorems~\ref{critical points}~and~\ref{Morse index}, we get
\begin{prop}[{\bfseries Critical~points~and~Morse~indices~in~the~`two~consecutive~beads~are~fixed'~case}]~\label{critical points in the two beads case}\\
Let $L > l$ and $n \ge 2$. Then
\begin{enumerate}
\item Critical points of $\A$ on $\M_{sm}((n, L), (1, l))$ are in bijection with the solutions of
$$
\left|U_{n - 1}(x)\right| = \dfrac{nl}{L}, \eqno (*)
$$
where $U_{n - 1}$ is the $(n - 1)$-th Chebyshev polynomial of second kind, that is, 
$
{U_{n - 1}(\cos \alpha) = \dfrac{\sin n\alpha}{\sin \alpha}.}
$ 
\item If $P$ is an admissible non-bifurcating critical configuration of $\A$ on $\M_{sm}((n, L), (1, l))$, then its Morse index is
$$
\mu_P^{n, 1}(\A) = \left\{\begin{aligned}
&2n - 2 - i, &\text{ if } &P \text{ corresponds to the $i$-th largest positive solution of } (*);\\
&i - 1, &\text{ if } &P \text{ corresponds to the $i$-th smallest negative solution of } (*).
\end{aligned}\right.
$$
\end{enumerate}
\end{prop}
\begin{proof}
By Theorem~\ref{critical points} a configuration $P \in \M_{sm}((n, L), (1, l))$ is a critical points of $\A$ if and only if it is inscribed in a circle $\Omega$ with centre $o$ and radius $R$ in such a way that ${\angle p_1op_2 = \ldots = \angle p_nop_{n + 1} =: \alpha}$. We set $c_P$ to be equal to $\cos(\alpha/2)$, where $\alpha/2 \in (0, \pi)$. Since $L/n = R\sqrt{2 - 2\cos \alpha} = 2R\sin(\alpha/2)$ and $l = R\sqrt{2 - 2\cos(n\alpha)} = 2R|\sin(n\alpha/2)|$, we get $U_{n - 1}(c_P) = nl/L$. The other direction is similar, so, we proved the first claim.

By symmetry reasons, to prove the second claim, it suffice to prove it only for $P$ with $c_P > 0$. Then by Theorem~\ref{Morse index} one has
$$
\mu_P^{n, 1}(\A) = (2n - 1) + \frac{1}{2}(\eps_{n + 1} + 1) - 1 - w_P - \left\{ \begin{aligned}
0, & \text{ if } n\tan(\alpha/2) > \eps_{n + 1}\tan(n\alpha/2);\\
1, & \text{ otherwise}.
\end{aligned} \right. 
$$
The roots and extrema of $U_{n - 1}(t)$ are interchanging. Lets start from $t = 1$ and move to the right. The extrema correspond to the bifurcating polygons (i. e. those with $n\tan(\alpha/2) = \eps_{n + 1}\tan(n\alpha/2)$ and the roots correspond to polygons with $l_{n + 1} = 0$. So, when $t$ passes a root, $\eps_{n + 1}$ changes from $1$ to $-1$ and whenever $t$ passes an extrema, the last summand changes from $0$ to $1$. When $p_1p_{n+1}$ passes through $o$,  $w_P$ increases by $1$, and $\eps_{n + 1}$ changes from $-1$ to $1$, which does not change the Morse index. The right-most $t$ corresponds to the global maximum, so the above argument completes the prove.
\end{proof}

Finally, we check that the last yet unproven ingredient of the proof of Theorem~\ref{Morse index} is in place.
\begin{proof}[Proof of Lemma~\ref{Morse index for cyclic}]
Let $P$ be as in the lemma. Without loss of generality we can assume that $\Omega_P = \Omega$ is the unit circle with center $o$, and, due to the symmetry reasons, it is enough to prove the statement for $P$ with $w_P > 0$. We should prove that the function 
$$
\frac{\A}{\L_1^2}\colon \left\{\text{polygons } P \text{ inscribed in the unit circle with } \frac{l_{n + 1}(P)}{\L_1(P)} = \frac{l}{L}\right\} \to \R
$$ 
attains a non-degenerate local maximum at $P$. For this it suffice to prove that the function
\begin{equation} \label{definition of G}
G\colon \left\{\begin{aligned}&\text{polygons inscibed}\\ &\text{in the unit circle} \end{aligned}\right\} \to \R, \q G(Q) = \dfrac{2\A(Q)}{l_{n + 1}(Q)^2} - \lambda \left(\frac{\L_1(Q)^2}{l_{n + 1}(Q)^2} - \frac{L ^ 2}{l^2}\right) - \mu \left(\frac{\L_1(Q)^2}{l_{n + 1}(Q)^2} - \frac{L ^ 2}{l^2}\right) ^ 2
\end{equation}
attains a non-degenerate local maximum at $P$ for suitable $\lambda$ and $\mu$.  We set ${\alpha = \angle p_1op_{2} = \ldots = \angle p_{n}op_{n + 1} \in (0, \pi)}$ and introduce local coordinates by setting ${t_i(Q) = \angle q_ioq_{i + 1} - \alpha}$ for ${i = 1, \ldots, n}$. First, we write the functions involved in the definition \eqref{definition of G} in these coordinates:
\begin{align*}
l_{n + 1}\left(t_1, \ldots, t_n\right) &= \sqrt{2 - 2\cos\left(n\alpha + \sum_{i = 1}^n t_i\right)};\\
\L_1\left(t_1, \ldots, t_n\right)  &= \sum_{i = 1}^n \sqrt{2 - 2\cos (\alpha + t_i)};\\
2\A\left(t_1, \ldots, t_n\right)  &= \sum_{i = 1}^n \sin(\alpha + t_i) - \sin\left(n\alpha + \sum_{i = 1}^n t_i\right)
\end{align*}
Second, we perform the computations in the 2-jets at point $P$, which by the aforementioned coordinates are identified with $\left.\R[t_1, \ldots, t_n]\right/I$, where $I$ is the ideal generated by all products $t_it_jt_h$ with $i, j, h = 1, \ldots, n$. It turns out that the 2-jets of the functions we are interested in are all contained in the subring $\R + \R T_1 + \R T_1 ^ 2 + \R T_2$, where $T_1 = \sum\limits_{i = 1}^n t_i$ and $T_2 = \sum_{i = 1}^n t_i ^ 2$. This subring is naturally identified with the ring $\mathcal R = \left.\R[T_1, T_2]\right/(T_1 ^ 3, T_2 ^ 2, T_1T_2)$. With all the identifications done, the 2-jets of the functions involved in the definition \eqref{definition of G} look as follows:
\begin{align*}
j_2l_{n + 1} &= l \cdot\left(1 + \frac{1}{2}\cot\left(\frac{n\alpha}{2}\right)T_1 - \frac{1}{8}T_1 ^ 2\right);\\
j_2\L_1 &= L \cdot\left(1 + \frac{1}{2n}\cot\left(\frac{\alpha}{2}\right)T_1 - \frac{1}{8n}T_2\right);\\
j_2(2\A) &= (n\sin \alpha - \sin(n\alpha)) + (\cos \alpha - \cos n\alpha) T_1 - \frac{\sin \alpha}{2}T_2 + \frac{\sin(n\alpha)}{2}T_1 ^ 2.
\end{align*}
Now, setting $x = \tan \frac{\alpha}{2}$ and $y = \tan \frac{n\alpha}{2}$, we can write the 2-jets of the summands in \eqref{definition of G} in more or less compact form:
\begin{align*}
j_2\left(\frac{\L_1 ^ 2}{l_{n + 1} ^ 2} - \frac{L ^ 2}{l ^ 2}\right) &= \frac{nx(1 + y ^ 2)(y - nx)}{y ^ 3(1 + x ^2)}T_1 - \frac{nx ^ 2(1+y^2)}{4y^2(1+x^2)}T_2 + C_1(n, x, y) T_1 ^ 2;\\
j_2\left(\frac{\L_1 ^ 2}{l_{n + 1} ^ 2} - \frac{L ^ 2}{l ^ 2}\right) ^ 2 &= \frac{n ^ 2 x ^ 2(1 + y ^ 2)^2(y - nx) ^ 2}{y ^ 6(1 + x ^ 2)}T_1 ^ 2;\\
j_2\left(\frac{2\A}{l_{n + 1} ^ 2} - \frac{2\A(P)}{l ^ 2}\right) &= \frac{(1 + y ^ 2)(y - nx)}{2y^3(1 + x ^2)}T_1 - \frac{x(1 + y ^ 2)}{4y^2(1 + x ^2)}T_2 + C_2(n, x, y) T_1^2.
\end{align*}
To get rid of $T_1$ in $j_2G$ we set $\lambda = \frac{1}{2nx}$, and then we finally obtain
$$
j_2\big(G - G(P)\big) = -\frac{x(1 + y ^ 2)}{8y^2(1 + x ^ 2)}T_2 + \left(C_2(n, x, y) - \frac{1}{2nx} C_1(n , x, y) - \mu \cdot \frac{n ^ 2 x ^ 2(1 + y ^ 2)^2(y - nx) ^ 2}{y ^ 6(1 + x ^ 2)}\right) T_1 ^ 2, 
$$
Note that the first summand is negative definite quadratic form since $x > 0$. As for the second term, $nx - y \ne 0$ as $P$ is not bifurcating, and thus, whatever $C_1$ and $C_2$ are, when $\mu$ is big enough the second term is non-positive definite quadratic form, hence $G$ attains a non-degenerate local maximum at $P$ for some large positive $\mu$ and we are done.
\end{proof}

\end{document}